\documentclass[a4paper,10pt,reqno]{amsart}

\usepackage[utf8]{inputenc}
\usepackage[T1]{fontenc}
\usepackage{amsthm}
\usepackage{amsmath}
\usepackage{amssymb}
\usepackage[inline]{enumitem}
\usepackage{comment}
\PassOptionsToPackage{hyphens}{url}\usepackage{hyperref}
\usepackage{fancyhdr}
\usepackage{mathrsfs}
\usepackage{stmaryrd}
\usepackage[normalem]{ulem}
\usepackage{xcolor}

\setlist[description]{%
  itemsep=0.05cm,            
  font={\normalfont\textsc}, 
  leftmargin=\parindent,
  labelindent=\parindent
}

\theoremstyle{definition}

\newtheorem{theorem}{Theorem}[section]
\newtheorem{lemma}[theorem]{Lemma}
\newtheorem{proposition}[theorem]{Proposition}

\theoremstyle{definition}

\newtheorem{conjecture}[theorem]{Conjecture}

\newtheorem{remark}[theorem]{Remark}

\newtheorem*{question*}{Question}

\newtheoremstyle{theoremdd}
  {\topsep}
  {\topsep}
  {\normalfont}
  {0pt}
  {\scshape}
  {:}
  { }
  {\indent\thmname{#1}\thmnumber{ #2}\textnormal{\thmnote{ (#3)}}}

\theoremstyle{theoremdd}
\newtheorem{claimx}{Claim}

\definecolor{blue-url}{RGB}{0,0,100}
\definecolor{red-url}{RGB}{100,0,0}
\definecolor{green-url}{RGB}{0,100,0}
\definecolor{light-yellow}{RGB}{255,255,128}
\definecolor{light-blue}{RGB}{193,255,255}
\definecolor{light-red}{RGB}{239,83,80}

\hypersetup{
	pdftitle={On the automorphisms of the power semigroups of a numerical semigroup},
	pdfauthor={},
	pdfmenubar=false,
	pdffitwindow=true,
	pdfstartview=FitH,
	colorlinks=true,
	linkcolor=blue-url,
	citecolor=green-url,
	urlcolor=red-url
}

\renewcommand{\qedsymbol}{$\blacksquare$}

\renewcommand{\setminus}{\smallsetminus}
\renewcommand{\,}{\kern 0.1em}

\providecommand\llb{\llbracket}
\providecommand\rrb{\rrbracket}

\providecommand\aut{{\rm Aut}}

\newcommand{\evid}[1]{\textsf{#1}}
\newcommand{\fin}{\mathrm{fin}}

\newlist{mydescription}{description}{1}
\setlist[mydescription]{
  leftmargin=2pt,    
  labelwidth=1cm,    
  itemindent=1.2cm,   
  align=left  
}

\begin{document}
\title{On the automorphisms \\ of the power semigroups of a numerical semigroup}
\author{Salvatore Tringali}
\address{(S.~Tringali) School of Mathematical Sciences, Hebei Normal University | Shijiazhuang, Hebei province, 050024 China}
\email{salvo.tringali@gmail.com}

\author{Kerou Wen}
\address{(K.~Wen) School of Mathematical Sciences, Hebei Normal University | Shijiazhuang, Hebei province, 050024 China}
\email{kerou.wen.hebnu@outlook.com}



\begin{abstract}
If $H$ is a numerical semigroup (that is, a cofinite subset of the
non-negative integers closed under addition), then the non-empty subsets of $H$ form a
semigroup $\mathcal P(H)$ under the sumset operation induced by addition in $H$.
Moreover, if $0 \in H$, then $\mathcal P(H)$ is a monoid with identity element
$\{0\}$, and the family $\mathcal P_0(H)$ of all subsets of $H$ containing $0$ is a
submonoid of $\mathcal P(H)$.

We show that the automorphism group of $\mathcal P(H)$ is trivial, and the same holds
for $\mathcal P_0(H)$ when $0 \in H$. 
The proofs blend ideas from combinatorics and semigroup theory.
\end{abstract}

\maketitle

\thispagestyle{empty}

\section{Introduction}
\label{sec:intro}

Let $H$ be a multiplicatively written semigroup (see the end of this section for notation and terminology). Endowed with the (binary) operation
of setwise multiplication induced by $H$, the collection of all
\textit{non-empty} subsets of $H$ forms a semigroup in its own right, hereinafter
denoted by $\mathcal P(H)$. The family $\mathcal P_{\fin}(H)$ of all non-empty
\textit{finite} subsets of $H$ is a sub\-semi\-group of $\mathcal P(H)$. More
explicitly, the operation is defined by
\[
XY := \{xy \colon x \in X,\, y \in Y\},
\qquad\text{for all } X, Y \in \mathcal P(H).
\]
We call $\mathcal P(H)$ and
$\mathcal P_{\fin}(H)$ the \evid{large power semigroup} and the \evid{finitary power semigroup} of $H$, respectively.

If, in particular, $H$ is a monoid with neutral element $1_H$, then $\mathcal P(H)$ and $\mathcal P_\fin(H)$ are themselves monoids, their
neutral element being the singleton $\{1_H\}$. Furthermore, the families
\[
\mathcal P_1(H) := \{X \subseteq H : 1_H \in X\}
\qquad\text{and}\qquad
\mathcal P_{\fin,1}(H) := \mathcal P_1(H) \cap \mathcal P_\fin(H)
\]
are submonoids of $\mathcal P(H)$ and $\mathcal P_\fin(H)$, respectively.
We refer to $\mathcal P_1(H)$ and $\mathcal P_{\fin,1}(H)$ as the \evid{reduced large power monoid} and the \evid{reduced finitary power monoid} of $H$, respectively.

\begin{remark}\label{rem:notation}
If the semigroup $H$ is written additively, we adopt the same convention for any subsemigroup of its large power
sem\-i\-group $\mathcal P(H)$. This amounts, in practice, to the fact that the operation
on $\mathcal P(H)$ maps a pair $(X,Y)$ of non-empty subsets of $H$ to their
\evid{sumset}
\[
X+Y := \{x+y : x \in X,\; y \in Y\},
\]
as opposed to the setwise product of $X$ by $Y$ that would be used in the multiplicative setting.
\end{remark}

The systematic study of power semigroups began in the late 1960s with the work
of Tamura and Shafer \cite{Tam-Sha1967, Shaf-1967}, although their earliest
explicit appearance in the literature dates back at
least to a 1950 paper by Ballieu \cite{bal-1950}. After a period of declining
activity, power semigroups resurfaced in a 2018 paper by Fan and Tringali
\cite{Fa-Tr18} and have since been investigated from several new perspectives,
revealing significant links with other areas, from arithmetic
combinatorics
\cite{Bie-Ger-22, Tri-Yan2025(b), Tri-Yan2025(a), Li-Tri2025, Rago26(b), Rago26}
to factorization theory
\cite{An-Tr18, GarSan-Tri2025(a), Cos-Tri2025, Gonz-Li-Rabi-Rodr-Tira-2025,
Agga-Gott-Lu-arXiv-21-12, Dani-et-al2025, Tri-2026(a)}. In recent years, special attention has been devoted to the automorphisms of
these structures.

For more details, let a \evid{numerical semigroup} be a cofinite subsemigroup $H$
of $\mathbb N$ (the additive monoid of non-negative integers). If $0 \in H$, we call $H$ a \evid{numerical monoid}.
Being among the simplest examples of finitely generated cancellative commutative semigroups,
numerical semigroups have been the subject of extensive research: they are intertwined with
classical questions in number theory (e.g., the Frobenius coin problem); they play an
important role in commutative algebra and algebraic geometry (e.g., in connection
with the study of coordinate rings); and they provide a
natural testing ground for conjectures and techniques in factorization theory and
semigroup theory. For additional information, the reader may consult the monograph
\cite{Ass-DAn-GarSan2020}, where, however, the term ``numerical semigroup''
is reserved for what we call a numerical monoid.

In the aftermath of their proof \cite[Theorem 2.5]{Tri-Yan2025(a)} of a conjecture
by Bienvenu and Geroldinger \cite[Conjecture 4.7]{Bie-Ger-22}, Tringali and Yan have established \cite[Theorem~3.2]{Tri-Yan2025(b)}
that the automorphism group of the reduced finitary power monoid $\mathcal P_{\fin,0}(\mathbb N)$ of the additive monoid of non-negative integers
is cyclic of order two.
As a follow-up,
Tringali and Wen \cite[Theorem~3.5]{Tri-Wen-2026(a)} have shown that the
automorphism group of the finitary power monoid of $\mathbb Z$
(the additive group of integers) is isomorphic to the direct product of a cyclic
group of order two and the infinite dihedral group, in stark contrast with the
fact that the only non-trivial automorphism of the same group is
the map $x \mapsto -x$. 
In a remarkable feat, Rago \cite[Theorem 13]{Rago26(b)} has subsequently proved that, for a finite abelian group $G$,
the automorphism group of $\mathcal P_{\fin,1}(G)$ is (naturally) isomorphic to $\aut(G)$,
except when $G$ is the Klein four-group, in which case $\aut(\mathcal P_{\fin,1}(G))$ is
isomorphic to the direct product of two copies of the symmetric group of degree three \cite[Example 2]{Rago26(b)}. Lastly, Wong et al.~\cite[Theorem 1.3]{Wong2025} have recently established that the automorphism group of the finitary power semigroup $\mathcal P_\fin(S)$ of a numerical semigroup $S$ is trivial unless $S = \llb k, \infty \rrb$ for some $k \in \mathbb N$, in which case $\aut(\mathcal P_\fin(S))$ is cyclic of order $2$. 

The present article adds to this line of research. One of its main novelties is
that we investigate, for the first time, automorphisms of \textit{infinitary} power
semigroups, that is, power semigroups involving both finite and infinite sets.
This setting gives rise to many new challenges, as several arguments (e.g., of a combinatorial nature) that work
in the finitary framework break down in a critical way.

\subsection*{Plan of the paper} Sec.~\ref{sect:02} contains some preliminaries that may be of independent
interest. In particular, we obtain in Theorem~\ref{thm:restriction-from-P1-to-Pfin1} that, for all \textit{cancellative} monoids $H$ and $K$, every isomorphism
from $\mathcal P_1(H)$ to $\mathcal P_1(K)$ restricts to an isomorphism from
$\mathcal P_{\fin,1}(H)$ to $\mathcal P_{\fin,1}(K)$. Furthermore, we show in Theorem~\ref{thm:restriction-from-P-to-Pfin} that, for all cancellative
\textit{commutative} semigroups $H$ and $K$, every isomorphism from
$\mathcal P(H)$ to $\mathcal P(K)$ restricts to an isomorphism from
$\mathcal P_\fin(H)$ to $\mathcal P_\fin(K)$. 


Secs.~\ref{sec:03} and \ref{sec:04} are devoted to our main results. We use Theorem~\ref{thm:restriction-from-P1-to-Pfin1} to show that the
automorphism group of the reduced large power monoid of a numerical monoid is
trivial (Theorem~\ref{thm:aut(P_0(H))}), and Theorem~\ref{thm:restriction-from-P-to-Pfin}
to show that the same holds for the large power semigroup of a numerical
semigroup (Theorem~\ref{thm:aut(P(S))-is-trivial}).

Finally, Sec.~\ref{sec:future} offers some directions for future research. Most
notably, we propose a couple of conjectures on how to extend the results of the present work far beyond the numerical setting.

The proofs blend ideas from combinatorics and semigroup theory, and we have striven to make them mostly self-contained. The only exceptions are
Lemma~\ref{(0,x) fixed} and
Lemma~\ref{lem:two-element-sets-are-fixed}.

\subsection*{Notation and terminology.} We use $\mathbb{N}$ for the additive monoid of non-negative integers, and $\mathbb{Z}$ for
the additive group of integers. For $a,b \in \mathbb{Z} \cup \{\pm \infty\}$, we let
$
\llb a, b \rrb := \{x \in \mathbb{Z} : a \le x \le b\}$ and $
\mathbb{Z}_{\ge a} := \llb a, +\infty \rrb$.

The \evid{Frobenius number} of a numerical semigroup $H$ is the maximum element of
the complement of $H$ in $\mathbb{Z}$ (this maximum exists because $H$ is, by
definition, a cofinite subset of $\mathbb{N}$).

We refer to \cite{Ho95} for background on semigroups and monoids. In particular,
we denote by $\aut(H)$ the automorphism group of a semigroup $H$ and, unless $H$
is a numerical semigroup or the power semigroup of a numerical semigroup (see also
Remark~\ref{rem:notation}), we write $H$ multiplicatively. An element $a \in H$ is
\evid{right} (resp., \evid{left}) \evid{cancellative} if the map
$H \to H \colon x \mapsto xa$ (resp., $H \to H \colon x \mapsto ax$) is injective.
The semigroup itself is \evid{cancellative} if each of its elements is right and
left cancellative.

Further notation and terminology, if not defined, is standard or should be clear
from the context. In particular, every homomorphism throughout this work is a \textit{semigroup} homomorphism (even if its domain and codomain are monoids), unless otherwise specified.

\section{Preliminary results}
\label{sect:02}

Below, we collect some preliminary results that will be needed in Secs.~\ref{sec:03} and \ref{sec:04} to prove the main theorems of the paper.
We begin with a couple of lemmas
implying that any isomorphism between the reduced large power monoids of two cancellative monoids restricts to an isomorphism between the corresponding reduced finitary power monoids (Theorem \ref{thm:restriction-from-P1-to-Pfin1}). 

\begin{lemma} \label{lem:AX=B_has_finitely_many_solutions}
Let $A$ and $B$ be non-empty subsets of a semigroup $H$. If $A$ contains a right cancellative element $a \in H$ and $B$ is finite, then the equation $XA = B$ has finitely many solutions $X$ over $\mathcal P(H)$.
\end{lemma}

\begin{proof}
If $XA = B$ for some $X \in \mathcal P(H)$, then $Xa \subseteq B$. On the other hand, the map $
\mathcal P(H) \to P(H) \colon X \mapsto Xa$ is injective, because $a$ is a right cancellative element of $H$, and hence $\{a\}$ is right cancellative in $\mathcal P(H)$ (see \cite[Remark~1]{Tri-2025(c)}
for details, if necessary). Since $B$ is finite (by hypothesis) and a finite set has finitely many subsets, it follows that the equation $XA = B$ has finitely many solutions $X$ over $\mathcal P(H)$.
\end{proof}

\begin{lemma}\label{XA=A^n}
Let $H$ be a monoid, $A$ be a subset of $H$ containing the identity $1_H$, and $n$ be an integer greater than $2$. Then the equation $XA=A^n$ has at least $2^{|A|-1}$ solutions $X$ over $\mathcal P_1(H)$.
\end{lemma}

\begin{proof}
Let $B$ be a subset of $A \setminus \{1_H\}$, and define $Q := A^{n-1} \setminus B$. We claim that $QA = A^n$.
Since $QA \subseteq \allowbreak A^{n-1}A = A^n$, we only need to check that $A^n$ is contained in $QA$. To this end, fix $z \in A^n$, and let $k$ be the smallest non-negative integer such that $z \in A^k$. Our goal is to show that $z \in QA$, and we distinguish three cases (note that $0 \le k \le n$; and since $1_H \in A$, we have $A \subseteq A^2 \subseteq \cdots \subseteq A^n $):

\vskip 0.2cm

\begin{itemize}[leftmargin=0.75cm]
\item \textsc{Case 1}: $k = 0$ or $k=1$. If $k=0$, then $z \in A^0=\{1_H\}$ and thus $z = 1_H$. So, regardless of whether $k=0$ or $k=1$, we find that $z \in A$, and hence $z \in 1_H A \subseteq QA$.

\vskip 0.2cm

\item \textsc{Case 2}: $2 \le k < n$. We have $z \notin B$, or else $z \in A$ (since $B \subseteq A$), contradicting the definition of $k$ (which guarantees that $z \notin A^i$ for every $i \in \llb 0, k-1 \rrb$). As a result, $z \in A^k \setminus B \subseteq Q$ (by the fact that $A^k \subseteq A^{n-1}$), and hence $z \in Q 1_H \subseteq QA$.

\vskip 0.2cm

\item \textsc{Case 3}: $k = n$. By the definition of $k$, we have $z \notin A^i$ for all $i \in \llb 0, n-1 \rrb$. On the other hand, $z = a_1 \cdots a_n$ for some $a_1, \ldots, a_n \in A$. It follows that $z' := a_1\cdots a_{n-1} \notin B$; otherwise, $z = z'a_n \in A^2$, contradicting $2 < 3 \le n$. Hence $z' \in A^{n-1} \setminus B$, and therefore $z = z' a_n \in QA$.
\end{itemize}
\vskip 0.1cm

Now, if $B_1$ and $B_2$ are distinct subsets of $A \setminus \{1_H\}$, then the sets $A^{n-1} \setminus B_1$ and $A^{n-1} \setminus B_2$ are also distinct. Consequently, the equation $XA = A^n$ has at least as many solutions $X$ in $\mathcal P_1(H)$ as there are subsets of $A \setminus \{1_H\}$, that is, it has at least $2^{|A|-1}$ solutions.
\end{proof}

\begin{theorem}
\label{thm:restriction-from-P1-to-Pfin1}
If $H$ and $K$ are cancellative monoids, then every isomorphism $f \colon \mathcal P_1(H) \to \mathcal P_1(K)$ maps $\mathcal P_{\fin,1}(H)$ bijectively onto $\mathcal P_{\fin,1}(K)$, and hence restricts to an isomorphism $\mathcal P_{\fin,1}(H) \to \mathcal P_{\fin,1}(K)$.
\end{theorem}

\begin{proof}
Let $A \in \mathcal P_{\fin,1}(H)$. Since $H$ is cancellative, we gather from Lemma \ref{lem:AX=B_has_finitely_many_solutions} that the equation $XA = \allowbreak A^3$ has finitely many solutions $X$ over $\mathcal P_1(H)$. Because $f$ is an isomorphism from $\mathcal P_1(H)$ to $\mathcal P_1(K)$, it follows that the equation $Yf(A)=f(A)^3$
has finitely many solutions $Y$ over $\mathcal P_1(K)$. On the other hand,
Lemma \ref{XA=A^n} guarantees that  the same equation has at least $2^{|f(A)|-1}$ solutions $Y$ over $\mathcal P_1(K)$. Consequently, $|f(A)|$ is finite, and hence $f(A) \in \allowbreak \mathcal{P}_{\mathrm{fin},1}(K)$. In other words, we have that $f[\mathcal{P}_{\fin,1}(H)]$ is contained in $\mathcal{P}_{\fin,1}(K)$, and for the reverse inclusion it is now sufficient to apply the same argument to the (functional) inverse $f^{-1}$ of $f$ and to consider that $f^{-1}$ is an isomorphism from $\mathcal P_1(K)$ to $\mathcal P_1(H)$.

In fact, since $K$ is also a cancellative monoid, we can apply the previous argument to $f^{-1}$. As a result, we find that $
f^{-1}[\mathcal{P}_{\fin,1}(K)] \subseteq \mathcal{P}_{\fin,1}(H)$, and therefore $
\mathcal{P}_{\fin,1}(K) \subseteq f[\mathcal{P}_{\fin,1}(H)]$.
\end{proof}

We continue with an analogue of Theorem~\ref{thm:restriction-from-P1-to-Pfin1} for isomorphisms between the large power semigroups of cancellative \textit{commutative} semigroups (Theorem~\ref{thm:restriction-from-P-to-Pfin}). It seems plausible that the same result could hold for cancellative semigroups in general, but proving this would likely require a totally different approach, as commutativity plays a critical role in the following preliminary lemma.

\begin{lemma}\label{at-least-|A|-solutions}
If $A$ is a non-empty subset of a cancellative commutative semigroup $H$ and $n$ is an integer greater than $1$, then the equation $A^{2n} = A^n X$ has at least $|A|$ solutions $X$ over $\mathcal P(H)$.
\end{lemma}

\begin{proof}
Since $A^{2n} = A^n A^n$, the claim is tantamount to proving that there exist at least $|A|-1$ sets $X \in \allowbreak \mathcal P(H)$ with $X \ne A^n$ such that $A^{2n} = A^n X$, which is obvious if $A$ is a singleton. Consequently, assume henceforth that $|A| \ge 2$. Next, fix $x \in A^{n-1}$ and write $x$ as $x_1 \cdots x_{n-1}$, where $x_1, \ldots, x_{n-1} \in A$.

Given $a \in A \setminus \{x_1\}$, we define $X_a := A^n \setminus \{ax\}$; note that the set $A \setminus \{x_1\}$ is non-empty (by the fact that $|A| \ge 2$) and $X_a$ is a proper subset of $A^n$ (by the fact that $ax \in A^n$). We claim $A^{2n} = A^n X_a$.

It is enough to check that $A^{2n} \subseteq A^n X_a$, as the reverse inclusion is trivial. To this end, let $u, v \in A^n$. If $u \ne ax$ or $v \ne ax$, then it is clear from the commutativity of $H$ that $uv = \allowbreak vu \in A^n X_a$. Otherwise,
\begin{equation*}
\begin{split}
uv & = ax_1 \cdots x_{n-1} \cdot ax_1 \cdots x_{n-1} = a^2 x_2 \cdots x_{n-1} \cdot x_1^2 x_2 \cdots x_{n-1} \in A^n X_a,
\end{split}
\end{equation*}
where we have used in particular that (i) both $a^2 x_2 \cdots x_{n-1}$ and $x_1^2 x_2 \cdots x_{n-1}$ are products of $n$ elements of $A$, and (ii) $x_1^2 x_2 \cdots x_{n-1} = x_1 x \ne ax$, since $H$ is cancellative and $a \ne x_1$ by construction.

In summary, we have shown that, for every $a \in A \setminus \{x_1\}$, the set $X_a$ is a proper subset of $A^n$ such that $A^{2n} = A^n X_a$. Since $X_b \neq X_c$ for all distinct $b, c \in A \setminus \{x_1\}$, the proof is thus complete.
\end{proof}

\begin{theorem}
\label{thm:restriction-from-P-to-Pfin}
Let $H$ and $K$ be cancellative commutative semigroups. Every isomorphism $f \colon \mathcal P(H) \to \mathcal P(K)$ restricts to an isomorphism $\mathcal P_\fin(H) \to \mathcal P_\fin(K)$, that is, maps $\mathcal P_\fin(H)$ bijectively onto $\mathcal P_\fin(K)$.
\end{theorem}

\begin{proof}
Let $A \in \mathcal P_\fin(H)$. Similarly as in the proof of Theorem \ref{thm:restriction-from-P1-to-Pfin1}, it suffices to verify that $f(A)$ is finite. To this end, we have from Lemma \ref{lem:AX=B_has_finitely_many_solutions} and the cancellativity of $H$ that the equation $A^2X=A^4$ has finitely many solutions $X$ over $\mathcal P(H)$.
Since $f$ is a bijective homomorphism from $\mathcal P(H)$ to $\mathcal P(K)$, it follows that the equation $   
f(A)^2Y=f(A)^4$
has finitely many solutions $Y$ over $\mathcal P(K)$. On the other hand, we get from Lemma \ref{at-least-|A|-solutions} and the hypotheses on $K$ that the same equation has at least $|f(A)|$ solutions $Y$ over $\mathcal P(K)$. Therefore, $|f(A)|$ is finite, and we are done.
\end{proof}

We conclude the section by establishing sufficient conditions under which an automorphism of $\mathcal P(H)$ or $\mathcal P_1(H)$, for an arbitrary monoid $H$, is trivial. 
We begin with a simple but crucial observation that allows us 
to convert a (purely) set-theoretic condition into an algebraic one.

\begin{remark}
\label{rem:algebraic-characterization-of-membership-for-idemps-containing-identity}
Let $H$ be a monoid and $E$ be a subset of $H$ such that 
$1_H \in E = E^2$. If an element $y \in H$ lies in $E$, then $E \subseteq \{1_H, y\}E \subseteq E^2 = E$. 
Conversely, if $\{1_H, y\} E = E$, then $y = y 1_H \in \{1_H, y\}E = E$. 
By symmetry, it follows that $y \in E$ if and only if $\{1_H, y\}E = E$, if and only if $E\{1_H, y\} = E$.
\end{remark}

The point of Remark \ref{rem:algebraic-characterization-of-membership-for-idemps-containing-identity}
is that, to the contrary of (certain) algebraic conditions, set-theoretic properties such as inclusion or element membership are, in general, 
not preserved by the automorphisms of a power monoid. We will immediately see this idea 
at work in the following:

\begin{proposition}\label{prop:equalizing-principle}
Let $H$ be a monoid, and let $\mathcal D$ be a submonoid of $\mathcal P(H)$ such that
\begin{enumerate}[label=(c\arabic{*})]
\item\label{prop:equalizing-principle(1)} $\{1_H, x\} \in \mathcal D$ for all $x \in H$, and
\item\label{prop:equalizing-principle(2)} every idempotent of $\mathcal D$ contains the identity $1_H$ of $H$.
\end{enumerate}
If an automorphism $f$ of $\mathcal D$ fixes each set $\{1_H, x\} \subseteq H$, 
then it fixes all the idempotents of $\mathcal D$.
\end{proposition}

\begin{proof}
Let $E$ be an idempotent of $\mathcal D$, and let $f$ be an automorphism of $\mathcal D$ 
such that $f(\{1_H, x\}) = \{1_H, x\}$ for every $x \in H$. We need to prove that $f(E) = E$. To this end, let $x$ be an arbitrary element of $H$. By 
Remark~\ref{rem:algebraic-characterization-of-membership-for-idemps-containing-identity}
and the hypotheses on $f$, we have that
\begin{equation}
\label{prop:equalizing-principle:equ(1)}
x \in E 
\quad\text{if and only if}\quad
\{1_H, x\}E = E, 
\quad\text{if and only if}\quad
\{1_H, x\}f(E)=f(E).
\end{equation}
On the other hand, $f(E)$ is itself an idempotent of $\mathcal D$, 
since semigroup homomorphisms are idempotent-preserving. 
By Condition~\ref{prop:equalizing-principle(2)}, it follows that $f(E)$ contains $1_H$. 
Hence,  
we conclude from Eq.~\eqref{prop:equalizing-principle:equ(1)} and 
Remark~\ref{rem:algebraic-characterization-of-membership-for-idemps-containing-identity} that $x \in E$ if and only if $x \in f(E)$, which in turn shows that $f(E) = E$.
\end{proof} 
\section{The automorphisms of \texorpdfstring{$\mathcal P_0(H)$}{P0(H)} when \texorpdfstring{$H$}{H} is a numerical monoid} 
\label{sec:03}

In this short section, we prove the first main result of the paper.
Before turning to the proof, we need two more lemmas, the first of which
relies, via the work of Tringali and Yan \cite{Tri-Yan2025(a)}, on a classical theorem of
Nathanson \cite{Nat78}, sometimes referred to as the Fundamental Theorem of Additive
Number Theory.

\begin{lemma}\label{(0,x) fixed}
If $H$ is a numerical monoid and $f$ is an automorphism of $\mathcal P_0(H)$, then $\{0,x\}$ is a fixed point of $f$ for every $x \in H$.  
\end{lemma}

\begin{proof}
Since numerical monoids are cancellative, Theorem \ref{thm:restriction-from-P1-to-Pfin1} guarantees that $f$ restricts to an automorphism of $\mathcal P_{\fin,0}(H)$. By \cite[Lemma 2.4]{Tri-Yan2025(a)}, this implies the
existence of an automorphism $g$ of $H$ such that $f(\{0, x\}) = \{0, g(x)\}$ for
all $x \in \allowbreak H$. However, it is folklore that the automorphism group of a numerical semigroup is trivial \cite[Theorem 3]{Higg-1969}. Hence, $g$ is the identity map on $H$, which in turn implies that
$\{0,x\}$ is a fixed point of $f$ for each $x \in H$.
\end{proof}

\begin{lemma}\label{H_y+X=H_y*+X}
Let $H$ be a numerical monoid, $X$ be a subset of $H$ containing $0$,
and $y$ be a non-zero element of $H$. Then $y \in X$ if and only if
$H_y + X = H_y^\ast + X$, where
\begin{equation}
\label{sec_01:eq(2)}
H_y := \{0\} \cup (H \cap \mathbb Z_{\ge y})
\quad\text{and}\quad
H_y^\ast := H_y \setminus \{y\}.
\end{equation}
\end{lemma}

\begin{proof}
Suppose first that $H_y + X = H_y^\ast + X$. From the fact that $y \in H_y$ and $0 \in X$, it is clear that $y \in \allowbreak H_y + X \allowbreak = H_y^\ast + \allowbreak X$.
Consequently, $
y = h + x$ for some $h \in H_y^\ast$ and $x \in X$. If $h \neq 0$, then $h > y$ and hence $
y = h + x > y$, which is impossible. It follows that $h = 0$, and thus $y = x \in X$.

Conversely, assume $y \in X$. Since $
H_y + X = (H_y^\ast + X) \cup (y + X)$,
it suffices to check that $y + X \subseteq \allowbreak H_y^\ast \allowbreak + X$. To this end, fix $x \in X$. We need to show that $y + x \in H_y^\ast + X$. If $x = 0$, then $y \in X \subseteq X + H_y^\ast$ (by the fact that $0 \in H_y^\ast$). Otherwise, $
y < y + x \in H_y^\ast \subseteq H_y^\ast + X$, as $H_y^\ast$ contains every element of $H$ larger than $y$ (and $X$ contains $0$). In either case, we are done.
\end{proof}

\begin{theorem}
\label{thm:aut(P_0(H))}
$\aut(\mathcal P_0(H))$ is trivial for every numerical monoid $H$.
\end{theorem}

\begin{proof}
Our goal is to show that the only automorphism $f$ of $\mathcal P_0(H)$ is the identity map. To start with, we know from Lemma \ref{(0,x) fixed} that $\{0,x\}$ is fixed by $f$ for all $x \in H$.
Together with Proposition \ref{prop:equalizing-principle}, this implies in turn that $f(E) = E$ for each idempotent $E \in \mathcal P_0(H)$.

Now, let $y \in H$ and $X \in \mathcal P_0(H)$ with $y \ne 0$. By Lemma \ref{H_y+X=H_y*+X}, we have that $y \in X$ if and only if $H_y + \allowbreak X = \allowbreak H_y^\ast + X$, where $H_y$ and $H_y^\ast$ are defined as in Eq.~\eqref{sec_01:eq(2)}.  
Since $f$ is a bijective endomorphism of $\mathcal P_0(H)$ and each of $H_y$ and $H_y^\ast$ is easily seen to be an idempotent of $\mathcal P_0(H)$, it follows by a further application of Lemma \ref{H_y+X=H_y*+X} that $y \in X$ if and only if $
H_y + f(X) = H_y^\ast + f(X)$, if and only if $y \in f(X)$. 

In other words, $f(X) = X$ for every $X \in \mathcal P_0(H)$, and hence $f$ is the identity map on $\mathcal P_0(H)$.
\end{proof}

\section{The automorphisms of \texorpdfstring{$\mathcal P(S)$}{P(S)} when \texorpdfstring{$S$}{S} is a numerical semigroup} 
\label{sec:04}

In this final section, we prove the second main result of the paper.
To begin, we establish a couple more lemmas,
the first of which relies on a theorem of Tringali \cite[Theorem~1]{Tri-2025(c)} that applies, in particular, to the large power semigroup of any
cancellative commutative semigroup.

\begin{lemma}\label{lem:two-element-sets-are-fixed}
If $H$ is a numerical semigroup and $f$ is an automorphism of $\mathcal P(H)$, then the singleton $\{x\}$ is a fixed point of $f$ for each $x \in H$.
\end{lemma}

\begin{proof}
Numerical semigroups are cancellative and commutative. Hence, we gather from \cite[Theorem~1]{Tri-2025(c)} that 
$f$ maps one-element sets bijectively onto one-element sets, that is, there exists an automorphism 
$g$ of $H$ such that $f(\{x\}) = \{g(x)\}$ for all $x \in H$. Since the automorphism group of $H$ 
is trivial (as already noted in the proof of Lemma \ref{(0,x) fixed}), it follows that every one-element subset of $H$ is fixed by $f$.
\end{proof}

\begin{lemma}\label{lem:minima-are-preserved}
Let $H$ be a numerical semigroup. 
If $f$ is an automorphism of $\mathcal P(H)$, then 
\[
\min f(X) = \min X,
\qquad\text{for every }
X \in \mathcal P(H).
\]
\end{lemma}

\begin{proof}
Let $k$ be an integer larger than the Frobenius number of $H$. Given $X \in \mathcal P(H)$, it is clear that
\begin{equation}
\label{equ:identity-with-semilines}
\mathbb Z_{\ge k} + X = \mathbb Z_{\ge k} + \min X = 
\mathbb Z_{\ge k + \min X}.
\end{equation}
Furthermore, the set $\mathbb Z_{\ge k}$ is an element of $\mathcal P(H)$, 
and Lemma \ref{lem:two-element-sets-are-fixed} ensures that the singleton
$\{\min X\}$ is a fixed point of $f$. Therefore, applying $f$ to 
Eq.~\eqref{equ:identity-with-semilines}, we obtain that
$$
f(\mathbb Z_{\ge k}) + f(X) = f(\mathbb Z_{\ge k}) + f(\{\min X\}) = f(\mathbb Z_{\ge k}) + \min X.
$$
This implies $\min f(\mathbb Z_{\ge k}) + \min f(X) = \min f(\mathbb Z_{\ge k}) + \min X$, and hence $\min f(X) = \min X$.
\end{proof}

The last ingredient we need is provided by the definition of a certain
(binary) relation that makes it possible to link the automorphism group
of the large power semigroup of a numerical semigroup
to that of the reduced large power monoid $\mathcal P_0(\mathbb N)$ of the
numerical monoid $\mathbb N$.

More precisely, let $Q$ be a \textit{commutative} semigroup (written
multiplicatively), and let $T$ be a subsemigroup of $Q$. Given $a,b \in Q$,
we say that $a$ is a \textsf{$T$-conjugate} of $b$ (in $Q$) if either
$a=b$ or $ax=by$ for some $x,y \in T$. Denote by $\sim_T$ the relation on
$Q$ defined by declaring $a \sim_T b$ if and only if $a$ is a
$T$-conjugate of $b$. It is routine to check that $\sim_T$ is a semigroup
congruence on $Q$, hereby referred to as the relation of \evid{$T$-conjugacy on $Q$}. Accordingly, we call the factor semigroup $Q/{\sim_T}$ the
\textsf{$T$-conjugate quotient} of $Q$, and the set
$\{y \in Q : x \sim_T y\}$ the \textsf{$\sim_T$-class} of an element
$x \in Q$.

When $Q$ is a (commutative) monoid and $T$ is the unit group of $Q$, the
quotient $Q/{\sim_T}$ is a familiar object of study in factorization
theory \cite[p.~xiii]{Ger-Hal-2006}. Here we focus instead on the case
when $Q$ is the large power semigroup of a numerical
semigroup and $T$ is the set of its cancellative elements. (It is a basic fact that the cancellative elements of a
semigroup form a subsemigroup.)

\begin{theorem}\label{thm:aut(P(S))-is-trivial}
$\aut(\mathcal P(H))$ is trivial for every numerical semigroup $H$.
\end{theorem}

\begin{proof}
Let $f$ be an automorphism of $\mathcal P(H)$ and $T$ be the set of cancellative elements of $\mathcal P(H)$, and denote by $\sim_T$ the relation of $T$-conjugacy on $\mathcal P(H)$.

\medskip

\begin{claimx}
\label{thm:aut(P(S))-is-trivial:claim_A}
The binary relation $\phi$ that maps the $\sim_T$-class $[X]$ of a non-empty subset $X$ of $H$ to $X - \allowbreak \min X$ is an iso\-mor\-phism from $\mathcal P(H)/{\sim_T}$ to $\mathcal P_0(\mathbb N)$.
\end{claimx}

\begin{proof}[Proof of Claim \ref{thm:aut(P(S))-is-trivial:claim_A}]
By \cite[Proposition~1]{Tri-2025(c)}, the cancellative elements of $\mathcal P(H)$ are precisely the sets $\{x\}$ with $x \in H$. It follows that if $X \sim_T Y$, then  $X+u=Y+v$ for some $u, v \in H$, and hence
$$
X - \min X = (X + u) - \min(X + u) = (Y+v)-\min(Y+v) = Y - \min Y.
 $$
Since $0 \in Z - \min Z \subseteq \mathbb N$ for all $Z \in \mathcal P(H)$, this shows that $\phi$ is actually a function (not just a binary relation) from $\mathcal P(H)/{\sim_T}$ to $\mathcal P_0(\mathbb N)$. In addition, we have 
\[
\phi([X]+[Y]) = \phi([X+Y]) = (X - \min X) + (Y - \min Y) = \phi([X]) + \phi([Y]), 
\]
which implies that $\phi$ is also a semigroup homomorphism. It remains to see that $\phi$ is bijective.

Indeed, if $\phi([X]) = \phi([Y])$ for some $X, Y \in \mathcal P(H)$, then $X + \min Y = Y + \min X$. Since $\min X$ and $\min Y$ are elements of $H$, it is thus clear that $[X] = [Y]$, and hence $\phi$ is injective. On the other hand, it is straightforward that $\phi$ is surjective, because for each $A \in \mathcal P_0(\mathbb N)$ we have $A + k \in \mathcal P(H)$ and $A = \phi([A + \allowbreak k])$ for every integer $k$ that is larger than the Frobenius number of $H$.
\renewcommand{\qedsymbol}{[\textit{Proof of \textsc{Claim} \ref{thm:aut(P(S))-is-trivial:claim_A}}\,]\,$\square$}
\end{proof}

\begin{claimx}
\label{thm:aut(P(S))-is-trivial:claim_B}
The binary relation $r_f$ that maps the $\sim_T$-class of a non-empty subset $X$ of $H$ to the $\sim_T$-class of $f(X)$ is an automorphism of $\mathcal P(H)/{\sim_T}$.
\end{claimx}

\begin{proof}[Proof of Claim \ref{thm:aut(P(S))-is-trivial:claim_B}] If $X \sim_T Y$, then $X + u = Y + v$ for some $u, v \in H$ and hence, by Lemma \ref{lem:two-element-sets-are-fixed}, $f(X) + \allowbreak u = f(Y) + \allowbreak v$, that is, $f(X) \sim_T f(Y)$. Similarly as in the proof of Claim \ref{thm:aut(P(S))-is-trivial:claim_A}, it follows that $r_f$ is actually a function on $\mathcal P(H)/{\sim_T}$. Moreover, $f$ being an endomorphism of $\mathcal P(H)$ implies
\[
r_f([X] + [Y]) = r_f([X+Y]) = [f(X+Y)] = [f(X) + f(Y)] = [f(X)] + [f(Y)] = r_f([X]) + r_f([Y]).
\]
That is, $r_f$ is an endomorphism of $\mathcal P(H)/{\sim_T}$. We are left to establish that $r_f$ is bijective.

We have from the above that the inverse $f^{-1}$ of $f$ induces a (well-defined) endomorphism of $\mathcal P(H)/{\sim_T}$ by mapping the $\sim_T$-class $[Z]$ of a non-empty set $Z \subseteq H$ to $[f^{-1}(Z)]$; and it is trivial to check that $r_{f^{-1}}$ is the inverse of $r_f$. Equivalently, $r_f$ is a bijection.
\renewcommand{\qedsymbol}{[\textit{Proof of \textsc{Claim} \ref{thm:aut(P(S))-is-trivial:claim_B}}\,]\,$\square$}
\end{proof}

Now, fix $X \in \mathcal P(H)$. We need to prove that $f(X)=X$.
By Claims~\ref{thm:aut(P(S))-is-trivial:claim_A}
and~\ref{thm:aut(P(S))-is-trivial:claim_B},
the composite function $\phi \circ \allowbreak r_f \circ \allowbreak \phi^{-1}$ is an automorphism
of $\mathcal P_0(\mathbb N)$. However, by
Theorem~\ref{thm:aut(P_0(H))}, the automorphism group of $\mathcal P_0(\mathbb N)$ is trivial. Consequently,
$\phi \circ r_f \circ \phi^{-1}$ is the identity map on
$\mathcal P_0(\mathbb N)$, and composing with $\phi^{-1}$ on the left
and with $\phi$ on the right yields that $r_f$ is in turn the identity map on the quotient
$\mathcal P(H)/{\sim_T}$.
It follows, by the definition of $r_f$, that $
[X]=r_f([X])=[f(X)]$,
that is,
\[
X+u=f(X)+v,
\qquad\text{for some } u, v\in H.
\]
On the other hand, Lemma~\ref{lem:minima-are-preserved} guarantees that
$\min X=\min f(X)$, which implies from the last displayed equation that
$u=v$, and hence $X=f(X)$, as desired.
\end{proof}

\section{Prospects for future research}
\label{sec:future}

We have found that the automorphism group of the reduced large power monoid of a
numerical monoid is trivial (Theorem~\ref{thm:aut(P_0(H))}), and that the same
holds for the automorphism group of the large power semigroup of a numerical
semigroup (Theorem~\ref{thm:aut(P(S))-is-trivial}). There is some evidence
that such ``rigidity phenomena'' may extend far beyond the numerical setting.
In this direction, several of the results established in Sec.~\ref{sect:02}
appear likely to play a useful role in future investigations.

More concretely, let $H$ be a semigroup, and let $\Phi$ be the function $\aut(H) \mapsto \aut(\mathcal P(H))$ that sends an automorphism $g$ of $H$ to its \evid{augmentation}, namely, the automorphism $g^\ast$ of $\mathcal P(H)$ given by
\[
g^\ast(X) := \{g(x) : x \in X\},
\qquad\text{for all }
X \in \mathcal P(H).
\]
It is easily verified that $\Phi$ is an injective group homomorphism, see, for instance, \cite[Sec.~1]{GarSan-Tri2025(a)}. However, in general, $\Phi$ need not be an isomorphism. Accordingly, we call an automorphism $f$ of $\mathcal P(H)$ \evid{inner} if it is the augmentation of an automorphism of $H$ (that is, if $f$ belongs to the image of $\Phi$), and we say that $\aut(\mathcal P(H))$ is \evid{canonically isomorphic} to $\aut(H)$ when $\Phi$ is surjective.

If, on the other hand, $H$ is a monoid, then the augmentation of an automorphism of $H$ restricts to an automorphism of $\mathcal P_1(H)$, and hence $\Phi$ induces an injective group homomorphism $\Phi_1 \colon \aut(H) \to \aut(\mathcal P_1(H))$. Again, we say that an automorphism of $\mathcal P_1(H)$ is inner if it belongs to the image of $\Phi_1$.

\begin{conjecture}\label{conj:A}
The automorphisms of the large power monoid $\mathcal P(\mathbb Z)$ of $\mathbb Z$ (the additive group of integers) are all inner, and so are the automorphisms of the reduced large power monoid $\mathcal P_0(\mathbb Z)$ of $\mathbb Z$. 
\end{conjecture}

In particular, it would follow from Conjecture \ref{conj:A} that $\aut(\mathcal P(\mathbb Z))$ and $\aut(\mathcal P_0(\mathbb Z))$ are both cyclic groups of order two, since the only non-trivial automorphism of $\mathbb Z$ is the map $x \mapsto -x$. 

There is in fact some evidence that Conjecture \ref{conj:A} holds in much greater generality for large classes of cancellative commutative monoids. A key step in this direction would be provided by the following:

\begin{conjecture}\label{conj:B}
If $H$ is a cancellative monoid, then every automorphism of $\mathcal P(H)$ maps $\mathcal P_1(H)$ bijectively onto itself and therefore restricts to an automorphism of $\mathcal P_1(H)$.
\end{conjecture}

If Conjecture~\ref{conj:B} holds, then the problem of determining the automorphism
group of $\mathcal P(H)$ for a cancellative monoid $H$ would be
naturally linked, by Theorem~\ref{thm:restriction-from-P1-to-Pfin1}, to the
problem of determining the automorphism group of $\mathcal P_{\fin,1}(H)$ and
hence to the Bienvenu--Geroldinger conjecture
\cite{Tri-Yan2025(a), Rago26, Tri-Yan2026(a)}.

\section*{Acknowledgments}

The authors were both supported by the Natural Science Foun\-da\-tion of Hebei Province through grant A2023205045. They are grateful to an anonymous referee for their careful reading and helpful comments on an earlier version of this manuscript.

\end{document}